\documentclass{amsart}

\usepackage{fullpage}
\usepackage{amsmath,amssymb,amsthm}
\usepackage{tikz}
\usepackage{stmaryrd,mathrsfs}
\usepackage{prettyref,hyperref}
\usetikzlibrary{matrix,calc,3d}

\newcommand{\A}{\hat{\mathbb{A}}}
\newcommand{\G}{\hat{\mathbb{G}}}
\renewcommand{\epsilon}{\varepsilon}
\newcommand{\Z}{\mathbb{Z}}
\newcommand{\Q}{\mathbb{Q}}
\newcommand{\F}{\mathbb{F}}
\newcommand{\N}{\mathbb{N}}
\newcommand{\C}{\mathbb{C}}
\newcommand{\CP}{\C\mathrm{P}}
\newcommand{\fc}[2]{\zeta{}^{#1}_{#2}}
\newcommand{\<}{\langle}
\renewcommand{\>}{\rangle}

\newcommand{\isoto}{\xrightarrow{\cong}}
\newcommand{\into}{\hookrightarrow}

\newcommand{\CatOf}[1]{\mathsf{#1}}
\newcommand{\torsor}[1]{\mathscr{#1}}
\newcommand{\sheaf}[1]{\mathcal{#1}}

\newcommand{\drawbox}[1]{\begin{array}{|l|}\hline #1 \\ \hline\end{array}}

\DeclareMathOperator{\spec}{Spec}
\DeclareMathOperator{\spf}{Spf}
\DeclareMathOperator{\Hom}{Hom}
\DeclareMathOperator{\im}{im}
\DeclareMathOperator{\Ext}{Ext}
\DeclareMathOperator{\Aut}{Aut}
\DeclareMathOperator{\id}{id}

\newtheoremstyle{thmstyle}{\topsep}{0.5em}{}{}{\bf}{:}{0.5em}{}
\theoremstyle{thmstyle}
\newtheorem{thm}{Theorem}[section]
\newtheorem{defn}[thm]{Definition}
\newtheorem{cor}[thm]{Corollary}
\newtheorem{lem}[thm]{Lemma}
\newtheorem{rem}[thm]{Remark}

\newtheorem{ithm}{Theorem}

\newrefformat{lem}{Lemma~\ref{#1}}
\newrefformat{def}{Definition~\ref{#1}}
\newrefformat{sec}{Section~\ref{#1}}
\newrefformat{thm}{Theorem~\ref{#1}}
\newrefformat{rem}{Remark~\ref{#1}}
\newrefformat{cor}{Corollary~\ref{#1}}
\newrefformat{fig}{Figure~\ref{#1}}

\title{Multiplicative 2-cocycles at the prime 2}
\author{Adam Hughes, JohnMark Lau, Eric Peterson
}

\begin{document}

\maketitle

\begin{abstract}
Using a previous classification result on symmetric additive $2$-cocycles, we collect a variety of facts about the Lubin-Tate cohomology of formal groups to compute the $2$-primary component of the scheme of symmetric multiplicative $2$-cocycles.  This scheme classifies certain kinds of highly symmetric multiextensions, as studied in general by Mumford or Breen.  A low-order version of this computation has previously found application in homotopy theory through the $\sigma$-orientation of Ando, Hopkins, and Strickland, and the complete computation is reflective of certain structure found in the homotopy type of connective $K$-theory.
\end{abstract}

\section{Introduction}

Biextensions and cubical structures are classically studied objects in algebraic geometry, appearing most prominently perhaps in the work of Mumford~\cite{Mumford} in the context of understanding Weil pairings.  The moduli problem associated to these objects appears as a low order of a family of moduli spaces associated to ``$\Theta^k$-structures,'' as studied by Breen~\cite{Breen}.  These low-order moduli objects have been previously studied in the context of homotopy theory by Ando, Hopkins, and Strickland~\cite{AHS,AS}, where they appear as the functors represented by the homology of certain connective covers of complex $K$-theory.  They go on to demonstrate that this computation for $k = 3$ gives rise to a canonical multiplicative map $MU\<6\> \to E$ for any elliptic spectrum $E$, and so analysis of this problem fits into the broader program of understanding the geometry controlling elliptic cohomology.  We continue their program by analyzing these moduli objects for $k > 3$ and to outline the apparent connection to $BU \< 2k \>$.

In order to exhibit these maps, it is necessary to explicitly compute the ring of functions on this moduli space, which they accomplish for $k \le 3$ by first computing an approximating scheme and then making an extensive case analysis of that data.  In a previous paper~\cite{HLP}, we constructed a similar approximating scheme for arbitrary $k$ and computed it in full.  In this paper, we compute the $2$-primary component of the original moduli objects of $\Theta^k$-structures for arbitrary $k$.  The main result is:
\begin{ithm} We make the computation
\begin{align*}
\sheaf{O}(\spec \Z_{(2)} \times C^k(\G_a; \G_m)) & = \Z_{(2)}[z_n \mid \nu_2 \phi(n, k) \le \nu_2 n] \otimes \\
& \otimes \Gamma[b_{n, \gamma_2(n, k)} \mid \nu_2 \phi(n, k) > \nu_2 n] \otimes \\
& \otimes \Z_{(2)}[b_{n, i} \mid \gamma_2(n, k) < i < D_{n, k}] / \<2b_{n,i}, b_{n,i}^2\>.
\end{align*}
where $D_{n, k}$ is the coefficient of the generating function \[\prod_{i=0}^\infty \frac{1}{1 - tx^{2^i}} = \sum_{n, k} D_{n, k} x^n t^k,\] $\phi(n, k)$ is defined by \[\phi(n, k) = \gcd_{\lambda} \binom{|\lambda|}{(\lambda_1, \ldots, \lambda_k)} = \gcd_{\lambda} \left( n! \prod_i (\lambda_i!)^{-1} \right),\] $\nu_p(n)$ denotes the number of times into which $p$ divides $n$, and $\gamma_p(n, k)$ denotes the function $\gamma_p(n, k) = \max\{0,\min\{k-\sigma_p(n),\nu_p(n)\}\}$, where $\sigma_p(n)$ denotes the $\N$-valued $p$-adic digital sum of $n$.
\end{ithm}

This result is proven by building an obstruction theory.  First, we recall various key definitions from algebraic geometry, including that of formal schemes, Lubin-Tate cohomology of formal groups, multiextensions, and (higher) cubical structures.  Then, we compute the tangent spaces $T_1 H^*(F; G)$ and $T_1 C^*(F; G)$ for our formal groups $F = \G_a$ and $G = \G_m$; this first calculation is done in the style of Hopkins and the second is the content of our previous paper.  The cohomological calculation is then used as input to a ``tangent spectral sequence'' \[T_1 H^*(F; G) \Rightarrow H^*(F; G),\] where we produce a family of nonvanishing differentials on certain key classes.  Finally, we recall some geometry related to Weil forms, which is certainly known to experts, but does not appear to be available in the literature; the important result for us is the existence of a certain $(k-1)$-variate cocycle $e$, which we call the half-Weil pairing, associated to any $k$-variate cocycle $u \in C^k(F; G)$.  Together they satisfy the two relations
\begin{align*}
\delta_1 e & = u, & e = \prod_{i=1}^{p-1} u(ix_1, x_1, \ldots, x_{n-1}).
\end{align*}
When applied to cocycles $u \in \spec \F_2 \times C^k(\G_a; \G_m)$, the cohomology class $[e_+]$ of the image of $e$ in the tangent space $T_1 C^k(\G_a; \G_m)$ belongs to the sources of our family of differentials in the tangent spectral sequence, and hence is obstructed from becoming $e$ (which in turn are obstructed from satisfying $u = \delta_1 e$) unless certain conditions are met --- namely that the leading coefficient square to zero.  We are then able to read off which classes in $T_1 C^k(\G_a; \G_m)$ extend unobstructed and which become obstructed, which culminates in the description of $\sheaf{O}(\spec \F_2 \times C^k(\G_a; \G_m))$.  This then immediately yields a description of $\sheaf{O}(\spec \Z_{(2)} \times C^k(\G_a; \G_m))$ as a consequence of the previous calculation of $\spec \Z_{(2)} \times C^k(\G_a; \G_a)$.  These techniques also produces partial data at odd primes, a phenomenon we discuss in closing.

\subsection{Acknowledgements}

This project was funded in part by NSF grant DMS-0705233.  Though $H^*(\G_a; \G_a)$ has been known for some time, the style of computation presented in \prettyref{thm:HGaGa} is difficult to find in the literature outside of some unpublished course notes.  We have heard it attributed to Mike Hopkins, who we too credit with thanks.  A number of the ideas in this paper were recycled from an early preprint of the Ando-Hopkins-Strickland paper; they all deserve special thanks.  Finally, Matt Ando deserves even \emph{more} special thanks for suggesting and supervising the project and for his infinite patience with all our questions.

\section{Formal groups}

\begin{defn}
Fix a commutative ring $R$ with unit, and consider the category $\CatOf{Algebras}_{/R}$ of augmented $R$-algebras, complete and separated in the adic topology induced by powers of their augmentation ideal, with continuous, unity-preserving algebra homomorphisms.  The category of ``formal schemes over $\spec R$'' is defined as the following category of presheaves: \[\CatOf{FormalSchemes}_{/\spf R} = \Hom(\CatOf{Algebras}_{/R}, \CatOf{Sets}).\]  The Yoneda embedding $A \mapsto \CatOf{Algebras}_{/R}^{cts}(A, -)$ is denoted $\spf A$.
\end{defn}

\begin{lem}
This category has various nice properties:
\begin{enumerate}
\item This category is cocomplete.
\item This category has an internal hom-object; define \[\underline{\Hom}(X, Y)(S) = \{(u, f) \mid u: \spf S \to \spf R, f: u^* X \to u^* Y\}.\]
\item This satisfies the exponential relation \[\Hom(X \times_{\spf R} Y, Z) \cong \Hom(X, \underline{\Hom}(Y, Z)).\]
\end{enumerate}
\end{lem}
\begin{proof}
A good reference for these facts --- indeed, for this entire section --- is Strickland~\cite{Strickland}.
\end{proof}

\begin{defn}
Formal affine $n$-space is defined to be $\A^n = \spf R \llbracket x_1, \ldots, x_n \rrbracket$.  A formal variety $V$ is a formal scheme noncanonically isomorphic to $\A^n$ for some $n$.  A coordinate on $V$ is a selected such isomorphism $\A^n \isoto V$.  For our purposes, a formal group is a commutative group object in the category of formal varieties which is isomorphic to $\A^1$.
\end{defn}

\begin{defn}
The Lubin-Tate cochains (see Lubin and Tate~\cite{LubinTate}) of a pair of formal groups $(F; G)$ is defined as \[A^n(F; G) = \underline{\Hom}(F^{\times n}, G).\]  There is a structure of cosimplicial object on this coming from the group operation in $F$; hence there are maps $\delta^n: A^n(F; G) \to A^{n+1}(F; G)$ forming a cochain complex upon evaluation on an algebra $S$.  The kernel of $\delta^n$ is the group of Lubin-Tate $n$-cocycles , denoted $Z^n(F; G)$.

There is a collection of formal schemes $H^*(F; G)$ whose action is \[H^n(F; G)(S) = \frac{\ker \delta^n: Z^n(F; G)(S) \to Z^{n+1}(F; G)(S)}{\im \delta^{n-1}: Z^{n-1}(F; G)(S) \to Z^n(F; G)(S)}.\] 
\end{defn}

\begin{defn}
The Lubin-Tate $k$-variate symmetric $2$-cocycle group $C^k(F; G)$ is a subscheme of $A^k(F; G)$ consisting of points $f: F^{\times k} \to G$ with $f(\sigma x) = f(x)$ and \[f(x1, x2, x3, \ldots) -_G f(x_0 +_F x_1, x_2, x_3, \ldots) +_G f(x_0, x_1 +_F x_2, x_3, \ldots) -_G f(x_0, x_1, x_3, \ldots) = 1_G.\]
\end{defn}

\begin{rem}
It is possible, though messy, to phrase construct all of these objects, including $C^k(F; G)$, using coordinate-free techniques.  Since we intend to calculate things, we'll be using coordinates anyway, and these coordinate-ful definitions are not really a disadvantage.
\end{rem}

\begin{defn}
The formal scheme $\spf R[\epsilon] / \epsilon^2$ plays the role of a point equipped with a tangent vector in the language of formal schemes.  The tangent bundle $TX$ of a scheme $X$ is then defined as \[TX = \underline{\Hom}(\spf R[\epsilon] / \epsilon^2, X).\]  Note then that
\begin{align*}
TX(S) & = \Hom(\spf S, \underline{\Hom}(\spf R[\epsilon]/\epsilon^2, X)) \\
& = \Hom(\spf S \times_{\spf R} \spf R[\epsilon] / \epsilon^2, X) \\
& = \Hom(\spf S[\epsilon] / \epsilon^2, X).
\end{align*}
Given an $R$-valued point $x: \spf R \to X$, the tangent space at $x$ is the subscheme of $TX$ restricting to $x$ along the map $\spf R \to \spf R[\epsilon] / \epsilon^2$ induced by $\epsilon \mapsto 0$.  When $X$ is a group scheme, we write $T_1 X$ for the tangent space of $X$ at the identity point.
\end{defn}

\begin{defn}
The most important formal groups in this paper are $\G_a$ and $\G_m$, both isomorphic to $\A^1$ as varieties.  The functor $\G_a$ is described on an $I$-adic $R$-algebra $A$ by $\G_a(R) = I$ with group law \[x +_{\G_a} y = x + y.\]  The functor $\G_m$ is described on $R$ by $\G_m(R) = 1 + I$ with group law given by \[(1+x) +_{\G_m} (1+y) = (1 + x)(1 + y) = 1 + (x + y + xy).\]  The isomorphism $\G_m \cong \A^1$ is given by $1 + x \mapsto x$, and so the group law induced on the formal affine line is described by $x + y + xy$.
\end{defn}

\begin{lem}
$T_1 G \cong \G_a$.
\end{lem}
\begin{proof}
Both $T_1 G$ and $\G_a$ are isomorphic to $\A^1$; the points $\epsilon a \in T_1 \G_m(R)$ are in bijective correspondence with the points $a \in \G_a(R)$.  Moreover, this map respects the group laws, since every formal group law is of the form $x +_G y = x + y + o(2)$.
\end{proof}

\section{Torsors and the moduli of extensions}

For ordinary groups $F$ and $G$, the cohomology groups $H^*(F; G)$ classify certain kinds of extensions of $G$ by $F$.  We reinterpret these ideas in the language of formal groups.  A good reference for the ideas in the first part of this section is Demazure and Gabriel~\cite{DemazureGabriel}.

\begin{defn}\label{def:Torsors}
Fix group schemes $G$ and $H$ and base $S$-schemes $X$ and $Y$.  Then a $G$-torsor $\torsor{L}$ over the scheme $X$ is an $X$-scheme with a $G$-action which is noncanonically $G$-equivariantly isomorphic to $G \times_S X$ as $X$-schemes.  That is, the $G$-action map fits into a map $G \times \torsor{L} \to \torsor{L} \times \torsor{L}$ described by $(g, x) \mapsto (gx, x)$, and this is an isomorphism of $S$-schemes.  If $\torsor{L}$ is as above and $\torsor{M}$ is an $H$-torsor over $Y$, then a map $\torsor{L} \to \torsor{M}$ of torsors is defined to be a pair of functions $G \to H$ and $X \to Y$ which commute with all the data present.
\end{defn}

\begin{rem}\label{rem:CohomologyClassifies}
The Lubin-Tate cohomology groups $H^*(F; G)$ classify group extensions of $F$ by $G$, i.e., $G$-torsors over $F$ with a chosen basepoint.  This is standard homological algebra hidden beneath the functorial veneer; a sequence of schemes $0 \to G \to Y \to F \to 0$ is said to be an extension of $F$ by $G$ when the evaluation at $A$ gives an extension of groups $F(A)$ by $G(A)$: \[0 \to G(A) \to Y(A) \to F(A) \to 0.\]  Since $F$ and $G$ are one-dimensional formal varieties, evaluation on $A = k \llbracket x, y \rrbracket$ gives $F(A) \cong k \llbracket x, y \rrbracket$ and $G(A) \cong k \llbracket x, y \rrbracket$, which in turn requires that $Y(A)$ be isomorphic to $k \llbracket x, y \rrbracket \times k \llbracket x, y \rrbracket$.  Adding the elements $(0, x)$ and $(0, y)$ in $Y(A)$ gives \[(0, x) +_{Y(A)} (0, y) = (u(x, y), x + y),\] with $u(x, y)$ some power series element of $G(A)$ satisfying the symmetry and $2$-cocycle conditions.  Viewing $u(x, y)$ as a map $k \llbracket z \rrbracket \to k \llbracket x, y \rrbracket$ and hence as a map $\A^2 \to \A^1$, we use our coordinates on $F$ and $G$ to produce a map $F^2 \to G$, and hence an element of $Z^2(F; G)$.  Using standard techniques, it can be shown that $u$ is determined up to $1$-coboundaries through change of coordinates, and hence $H^2(F; G) \cong \Ext^1(F; G)$.  Similar statements can be made for higher $\Ext$s.
\end{rem}

\begin{defn}\label{def:TorsorOps}
Several common constructions for bundles translate to torsors.
\begin{itemize}
\item \emph{Pullback:} Let $\torsor L$ be a $G$-torsor over $Y$, and let $f: X \to Y$ be a map of schemes.  Then we define the pullback $f^* \torsor L$ to be the fiber product $\torsor L \times_X Y$, which is easily seen to be a $G$-torsor over $Y$.
\item \emph{Pushforward:} Let $\torsor L$ be as above, and let $\varphi: G \to H$ be a map of group schemes.  Then we define the pushforward torsor $\varphi_* \torsor L$ as the colimit of the diagram
\begin{center}
\begin{tikzpicture}[
        normal line/.style={-stealth},
    ]
    \matrix (m) [matrix of math nodes, 
         row sep=2em, column sep=3em,
         text height=1.5ex, 
         text depth=0.25ex]{
\torsor L \times G \times H & \torsor L \times H \\
\torsor L \times H \times H & \torsor L \times H, \\
    };
    \path[normal line]
        (m-1-1) edge node[above]{$\cdot \times \id$} (m-1-2)
                edge node[left]{$\id \times \varphi \times \id$} (m-2-1)
        (m-1-2) edge[style={double,-}] (m-2-2)
        (m-2-1) edge node[below]{\hspace{0.8em}$\id \times \cdot$} (m-2-2);
\end{tikzpicture}
\end{center}
corresponding to the Borel construction. 
\item \emph{Product:} Let $\torsor L$ be a $G$-torsor over $X$ and $\torsor M$ an $H$-torsor over $Y$.  Then there exists a product $G \times H$-torsor $\torsor L \times \torsor M$ over $X \times Y$ given by the scheme-theoretic product and component-wise action.
\item \emph{Dual:} Let $1$ denote the trivial $G$-torsor $G \times X$ over the $S$-scheme $X$.  Then any $G$-torsor $\torsor L$ over $X$ has a dual defined by $\torsor L^{-1} = \underline{\Hom}_G(\torsor L, 1)$.
\end{itemize}
\end{defn}

\begin{rem}
These constructions can be combined to give several others, including the tensor product of torsors.  If $\torsor L$ and $\torsor M$ are two $G$-torsors over $X$, then $\torsor L \otimes \torsor M = \Delta^* \mu_* (\torsor L \times \torsor M)$, where $\Delta: X \to X \times X$ is the diagonal map and $\mu: G \times G \to G$ is multiplication.
\end{rem}

\begin{rem}
Given two $G$-torsors $\torsor{L}$ and $\torsor{M}$ over the $S$-schemes $X$ and $Y$ respectively, a $\star$-map $\torsor{L} \to \torsor{M}$ is a pair $(f, t)$ of a map $f: X \to Y$ and a $G$-equivariant isomorphism of $S$-schemes $t: \torsor{L} \to f^* \torsor{M}$.  This produces a category of $G$-torsors over $S$-schemes.  This is the category and notion of ``map of torsors'' usually taken; our map of torsors defined in \prettyref{def:Torsors} is strictly weaker.
\end{rem}

\begin{defn}\label{def:Multiextensions}
Fix a structure group $G$.  We make a sequence of definitions leading up to that of a higher cubical structure:
\begin{itemize}
\item Select a family of groups $H_1, \ldots, H_n$.  A multiextension $\torsor L$ is a $G$-torsor over $H_1 \times \cdots \times H_n$ so that for any point $h_{\hat \imath} = (h_1, \ldots, \hat h_i, \ldots, h_n) \in H_1 \times \cdots \times \hat H_i \times \cdots \times H_n$ the corresponding pullback $f^* \torsor L$ along $f(h_i) = (h_1, \ldots, h_i, \ldots, h_n)$ gives an extension of group schemes of $H_i$ by $G$.  These extensions are controlled by a family of $2$-cocycles parametrized by the missing index $i$ and the points $h_{\hat \imath}$: \[u_i(h_{\hat \imath}) : H_i \times H_i \to G.\]
\item In the case $H_1 = \cdots = H_n$, we can impose various symmetry conditions on such a multiextension.  To begin, we have a family of morphisms $\sigma: H^n \to H^n$ corresponding to permutations $\sigma \in \Sigma_n$, and to each permutation $\sigma$ we can construct a map $\Delta_\sigma: H^{|\mathbf{n} / \<\sigma\>|} \to H^n$ that populates the $\sigma$-orbits of $H^n$ with diagonal values.  We necessarily have $\Delta_\sigma = \sigma \Delta_\sigma$, and hence $\Delta_\sigma^* \torsor L$ is canonically isomorphic to $(\sigma \Delta_\sigma)^* \torsor L$.  The most basic condition asserts that we fix a family of isomorphisms $\tau_\sigma$ of isomorphisms $\tau_\sigma: \sigma^* \torsor L \to \torsor L$ extending these given isomorphisms satisfying the coherence relations $\tau_{\sigma' \sigma} = (\sigma^* \tau_{\sigma'}) \tau_{\sigma}$.  A multiextension together with this symmetry data is called a symmetric multiextension.
\item The most extreme symmetry we can request is a torsor trivialization of $\torsor L$ such that the controlling cocycles satisfy \[u_i(h_1, \ldots, \hat h_i, \ldots, h_n)(h_i, h_{n+1}) = u_{\sigma i}(h_{\sigma 1}, \ldots, \hat h_{\sigma i}, \ldots, h_{\sigma n})(h_{\sigma i}, h_{\sigma(n+1)})\] for all choices of $\sigma \in \Sigma_{n+1}$.  Under these conditions, we can simply write $u(h_1, \ldots, h_{n+1})$ without ambiguity, since all interpretations of this symbol produce the same point in $G$.  A multiextension satisfying this condition is called a higher cubical structure.  The name ``cubical structure'' stems from previous work of Mumford~\cite{Mumford} and Breen~\cite{Breen} on the case $n = 2$. 
\end{itemize}
\end{defn}

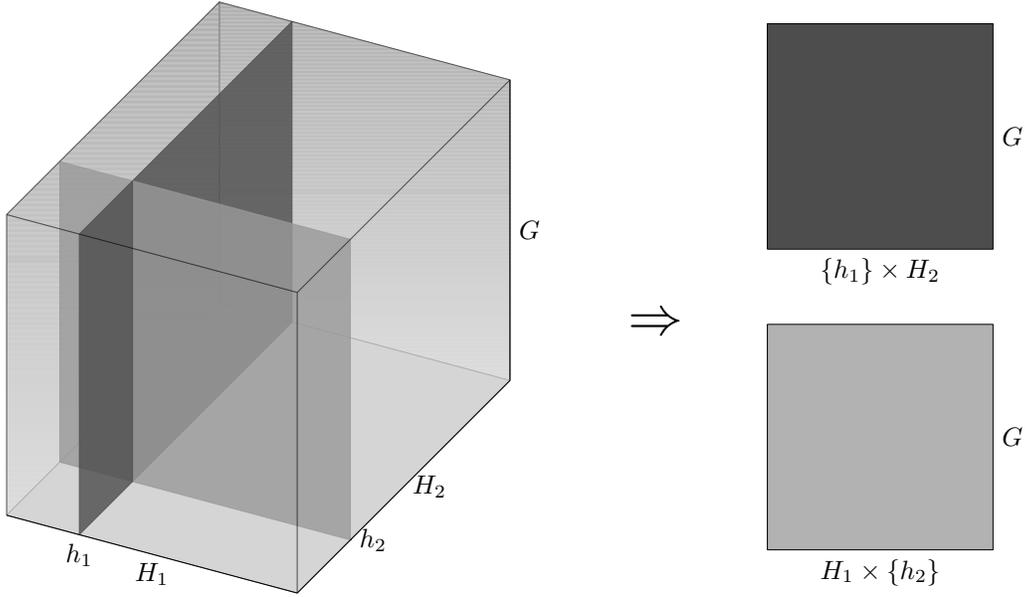
\begin{figure}[htp]
\begin{tikzpicture}[x  = {(-0.707cm,-0.707cm)},
                    y  = {(0.9659cm,-0.25882cm)},
                    z  = {(0cm,1cm)},
                    scale = 2,
                    color = {lightgray}]
\tikzset{facestyle/.style={fill=black!25,draw=black,very thin,line join=round,opacity=.4},slicestyle/.style={fill=black!25,draw=black,very thin,line join=round,opacity=.8}}
\begin{scope}[canvas is zy plane at x=0]
  \path[facestyle,shade] (0,0) rectangle (2,2);
\end{scope}
\begin{scope}[canvas is zx plane at y=0]
  \path[facestyle,shade] (0,0) rectangle (2,2);
\end{scope}
\begin{scope}[canvas is yx plane at z=0]
  \path[facestyle,color=lightgray] (0,0) rectangle (2,2);
\end{scope}
\begin{scope}[canvas is zy plane at x=1.5]
  \path[slicestyle,color=black!50] (0,0) rectangle (2,0.5);
\end{scope}
\begin{scope}[canvas is zx plane at y=0.5]
  \path[slicestyle,color=black] (0,0) rectangle (2,1.5);
\end{scope}
\begin{scope}[canvas is zy plane at x=1.5]
  \path[slicestyle,color=black!50] (0,0.5) rectangle (2,2);
\end{scope}
\begin{scope}[canvas is zx plane at y=0.5]
  \path[slicestyle,color=black] (0,1.5) rectangle (2,2);
\end{scope}
\begin{scope}[canvas is zy plane at x=2]
  \path[facestyle] (0,0) rectangle (2,2);
\end{scope}
\begin{scope}[canvas is zx plane at y=2]
  \path[facestyle] (0,0) rectangle (2,2);
\end{scope}
\begin{scope}[canvas is yx plane at z=2]
  \path[facestyle] (0,0) rectangle (2,2);
\end{scope}
\draw[very thin,black,line join=round]
     (2,0,0) -- node [below,black] {$H_1$} (2,2,0);
\draw[very thin,black,line join=round]
     (2,2,0) -- node [right,black] {$H_2$} (0,2,0);
\draw[very thin,black,line join=round]
     (0,2,0) -- node [right,black] {$G$} (0,2,2);
\draw[very thin,black,line join=round]
     (2,0.5,0) -- node [below,black] {$h_1$} (2,0.5,0);
\draw[very thin,black,line join=round]
     (1.5,2,0) -- node [right,black] {$h_2$} (1.5,2,0);
\end{tikzpicture}
\hspace{2em}
\begin{tikzpicture}
    \draw[very thin,black,line join=round,fill=black!30]
        (0,1) -- node [below,black] {$H_1 \times \{h_2\}$} (3,1) -- node[right,black] {$G$} (3,4) -- (0,4) -- (0,1);
    \draw[very thin,black,line join=round,fill=black!70]
        (0,5) -- node [below,black] {$\{h_1\} \times H_2$} (3,5) -- node[right,black] {$G$} (3,8) -- (0,8) -- (0,5);
    \node[anchor=east] at (-1,4) {\Huge $\Rightarrow$};
\end{tikzpicture}
\caption{Extensions contained in a biextension.}
\end{figure}

\begin{defn}
We define a sequence of functorial constructions of multiextensions.  Denote the map $(h_1, \ldots, h_n) \mapsto \sum_{i \in I} h_i$ by $\mu_I$, and select an extension $\torsor L$ of $H$ by $G$.  We define $\Theta^k \torsor L$ by
\begin{align*}
\Theta^k \torsor L &= \bigotimes_{I \subseteq \{1, \ldots, k\}} (\mu_I^* \torsor L)^{(-1)^{|I|}}, & (\Theta^k \torsor L)_{\mathbf{x}} & = \bigotimes_{I \subseteq \{1, \ldots, k\}} \torsor L_{\sum_{i \in I} \mathbf{x}_i}^{(-1)^{|I|}}.
\end{align*}
A $\Theta^k$-structure on a extension $\torsor L$ is a chosen trivialization of $\Theta^k \torsor L$.
\end{defn}

\begin{rem}
A $\Theta^{k+1}$-structure on $\torsor L$ corresponds to a higher cubical structure on $\Theta^k \torsor L$.  Both of these structures are classified by the Lubin-Tate cocycle groups $C^k(F; G)$.  Pick $k = 2$ for simplicity, and suppose we have a $\Theta^3$-structure on $\torsor L$, i.e., a selected isomorphism $\Theta^3 \torsor L \xrightarrow{\cong} 1$.  Then, we produce a map on fibers as follows:
\begin{align*}
1_{x, y, z} & \xrightarrow{\cong} \frac{\torsor L_{x + y + z} \otimes \torsor L_x \otimes \torsor L_y \otimes \torsor L_z}{\torsor L_{x + y} \otimes \torsor L_{x + z} \otimes \torsor L_{y+z}}, \\
\frac{\torsor L_{x + z} \otimes \torsor L_{y + z}}{\torsor L_x \otimes \torsor L_y \otimes \torsor L_z} & \xrightarrow{\cong} \frac{\torsor L_{x + y + z}}{\torsor L_{x + y}}, \\
\frac{\torsor L_{x + z}}{\torsor L_x \torsor L_z} \otimes \frac{\torsor L_{y + z}}{\torsor L_y \torsor L_z} & \xrightarrow{\cong} \frac{\torsor L_{x + y + z}}{\torsor L_{x + y} \otimes \torsor L_z}, \\
(\Theta^2 \torsor L)_{x, z} \otimes (\Theta^2 \torsor L)_{y, z} & \xrightarrow{\cong} (\Theta^2 \torsor L)_{x+y, z},
\end{align*}
part of the biextension structure.  The symmetry of the tensor product used in the definition of $\Theta^{k+1} \torsor L$ ensures that the induced multiextension structure on $\Theta^k \torsor L$ is a higher cubical structure.  Higher cubical structures are classified by their controlling cocycle $u$, i.e., a point of $C^k(F; G)$.  This observation is recounted in great, careful detail in both Breen~\cite{Breen} and Mumford~\cite{Mumford}.
\end{rem}

Finally, we make some remarks on how multiextensions interact with the torsor operations defined in \prettyref{def:TorsorOps}.

\begin{lem}
Let $\torsor B$ be a higher cubical structure with structure group an $S$-scheme $G$ over a group $S$-scheme $Y$, and select a map $f: X \to Y$ of group $S$-schemes.  Then the pullback $f^* \torsor B$ receives the structure of a symmetric multiextension so that the induced map $f^* \torsor B \to \torsor B$ is a map of multiextensions.  If $u$ is the controlling cocycle for $\torsor B$, then the controlling cocycles for $f^* \torsor B$ are $u \circ f$.
\end{lem}

\begin{lem}
Let $\torsor B$ be a higher cubical structure with structure group an $S$-scheme $G$ over an $S$-scheme $X$, and select a map $\varphi: G \to H$ of group $S$-schemes.  Then the pushforward $\varphi_* \torsor B$ receives the structure of a higher cubical structure so that $\torsor B \to \varphi_* \torsor B$ is a map of higher cubical structures.   If $u$ is the controlling cocycle for $\torsor B$, then the controlling cocycle for $\varphi_* \torsor B$ is described by $\varphi \circ u$.
\end{lem}

\section{Calculations tangent to the Lubin-Tate cohomology of $(\G_a; \G_m)$.}\label{sec:CalcTangent}

Our ultimate goal is to understand the group scheme $C^k(\G_a; \G_m)$.  As in Lie theory, it is fruitful to first compute the tangent space at the identity as a means of understanding the local picture.

\subsection{Calculation of $T_1 H^*(\G_a; \G_m)$}

Let's begin by computing the tangent space to the cohomology groups.

\begin{lem}\label{lem:T1GCommute}
$T_1 H^*(F; G) = H^*(F; T_1 G)$.
\end{lem}
\begin{proof}
We expand the definition of $H^*(F; G)$ to make the following calculation: \[T_1 H^n(F; G)(S) = \frac{\ker T_1 \delta^n: T_1 Z^n(F; G)(S) \to T_1 Z^{n+1}(F; G)(S)}{\im T_1 \delta^{n-1} : T_1 Z^{n-1}(F; G)(S) \to T_1 Z^n(F; G)(S)}.\]  Hence, we reduce to understanding $T_1 Z^{n+1}(F; G)$.

The point of $Z^n(F; G)$ corresponding to the identity element is represented by the power series $0$, sending $F^k$ to the identity point of $G$.  A point of $T_1 Z^n(F; G)$ is then a power series $u$ of the form $0 + \epsilon u_+$ for some power series $u_+$.  Since $\epsilon^2 = 0$, we compute the $G$-inverse of $\epsilon u_+$ to be $-\epsilon u_+$, and hence the $2$-cocycle condition on $u$ corresponds to the following condition on $u_+$: \[u_+(x_1, x_2, x_3, \ldots) - u_+(x_0 +_F x_1, x_2, x_3, \ldots) + u_+(x_0, x_1 +_F x_2, x_3, \ldots) - u_+(x_0, x_1, x_3, \ldots) = 0.\]  These $u_+$ are exactly the elements of $Z^n(F; T_1 G)$.  We also have inclusion in the other direction; a point $u_+ \in Z^n(F; T_1 G)$ corresponds to a point $0 + \epsilon u_+ \in T_1 Z^n(F; G)$.
\end{proof}
\begin{cor}
The tangent space $T_1 H^*(\G_a; \G_m)$ is $H^*(\G_a; \G_a)$.
\end{cor}

\begin{thm}\label{thm:HGaGa}
Let $a_i$ represent $x^{p^i}$ and $b_i$ represent $p^{-1}((x+y)^{p^i} - x^{p^i} - y^{p^i}) =: \fc{p^i}{2}$.  Then,
\begin{align*}
H^*(\G_a; \G_a)(\Q) & \cong \Lambda[b], \\
H^*(\G_a; \G_a)(\F_2) & \cong \bigotimes_i \F_2[a_i], \\
H^*(\G_a; \G_a)(\F_p) & \cong \left(\bigotimes_i \Lambda[a_i]\right) \otimes \left(\bigotimes_i \F_p[b_i]\right).
\end{align*}
\end{thm}
\begin{proof}
This is an exercise in homological algebra and application of the Tate resolution~\cite{Tate}.  The key ideas are that this boils down to computing $\Ext_{R[x]}(R, R)$ in the category of $R[x]$-comodules, that the P.D. algebra $\Gamma[x]$ is the dual of $R[x]$, and that this $\Ext$-calculation is isomorphic to $\Ext_{\Gamma[x]}(R, R)$ in the category of $\Gamma[x]$-modules.  From here, we split into cases:
\begin{enumerate}
\item $\Ext_{\Gamma_\Q[x]}(\Q, \Q)$: The chain complex
\begin{center}
\begin{tikzpicture}[
        normal line/.style={-stealth},
    ]
    \matrix (m) [matrix of math nodes, 
         row sep=2em, column sep=3em,
         text height=1.5ex, 
         text depth=0.25ex]{
0 & \Q[a] & \Q[b] & 0 \\
    };
    \path[normal line]
        (m-1-2) edge (m-1-1)
        (m-1-3) edge node[above]{$a \mapsfrom 1$} (m-1-2)
        (m-1-4) edge (m-1-3);
\end{tikzpicture}
\end{center}
is a projective resolution of $\Q$ with $\Q[a] \to \Q$ given by $a \mapsto 0$, so we compute $\Ext_{\Gamma_{\Q}[x]}(\Q, \Q)$ to be $\Lambda[b]$ as promised.
\item $\Ext_{\Gamma_{\F_2}[x]}(\F_2, \F_2)$: The algebra $\Gamma_{\F_2}[x]$ splits as the tensor product $\Gamma_{\F_2}[x] \cong \bigotimes_{i=0}^\infty \Lambda_{\F_2}[x^{[2^i]}]$, hence it suffices to compute $\Ext_{\Lambda_{\F_2}[y]}(\F_2, \F_2)$ and then tensor together those results.  The differential graded algebra described by Tate which computes $\Ext_{\Lambda[y]}(\F_2, \F_2)$ is given by $R_* = \Gamma[a] \cong \bigotimes_{i=0}^\infty \Lambda[a^{[2^i]}]$ with differential $da^{[j]} = a^{[j-1]}y$.  Therefore, \[\Ext_{\Gamma[x]}(\F_2, \F_2) \cong \bigotimes_{i=0}^\infty \Ext_{\Lambda[x^{[2^i]}]}(\F_2, \F_2) \cong \bigotimes_{i=0}^\infty \Hom(\Gamma[a_i], \F_2) \cong \bigotimes_{i=0}^\infty \F_2[a_i^\vee].\]
\item $H^*(\G_a; \G_a)(\F_p)$: Just as before, $\Gamma_{\F_p}[x]$ splits as a product of algebras $\Gamma_{\F_p}[x] \cong \bigotimes_{i=0}^\infty T[x^{[p^i]}]$, where $T[y]$ denotes the truncated polynomial algebra $T[y] = \F_p[y] / y^p$.  Hence, we reduce to calculating $\Ext_{T[y]}(\F_p, \F_p)$.  The Tate differential graded algebra is described by $\Lambda[a] \otimes \Gamma[b]$, where $da = y$ and $db^{[j]} = ab^{[j-1]}y^{p-1}$.  Therefore, \[\Ext_{\Gamma[x]}(\F_p, \F_p) \cong \bigotimes_{i=0}^\infty \Ext_{T[x^{[p^i]}]}(\F_p, \F_p) \cong \bigotimes_{i=0}^\infty \Hom(\Lambda[a_i] \otimes \Gamma[b_i], \F_p) \cong \bigotimes_{i=0}^\infty \Lambda[a_i^\vee] \otimes \F_p[b_i^\vee].\qedhere\]
\end{enumerate}
\end{proof}

\subsection{Calculation of $T_1 C^*(\G_a; \G_m)$}

The scheme $C^k(\G_a; \G_m)$ is also a group scheme, so also has a tangent space.

\begin{lem}
$T_1 C^k(F; G)$ is $C^k(F; T_1 G)$.
\end{lem}
\begin{proof}
This is identical to \prettyref{lem:T1GCommute}.
\end{proof}
\begin{cor}
$T_1 C^k(\G_a; \G_m) \cong C^k(\G_a; \G_a)$.
\end{cor}

\begin{defn}\label{def:zeta}
Let $\fc{n}{k}$ denote the integral polynomial \[\fc{n}{k} = \phi(n, k)^{-1} \sum_{\substack{X \subseteq \{x_1, \ldots, x_k\} \\ X \ne \emptyset}} \left( (-1)^{|X|} \cdot \left( \sum_{x \in X} x \right)^n    \right) = \phi(n, k)^{-1} \sum_{\substack{\lambda \vdash n \\ \ell(\lambda) = k}} \binom{n}{\lambda} \mathbf{x}^\lambda,\] where $\phi(n, k)$ is defined by \[\phi(n, k) = \gcd_{\lambda} \binom{|\lambda|}{(\lambda_1, \ldots, \lambda_k)} = \gcd_{\lambda} \left( n! \prod_i (\lambda_i!)^{-1} \right).\]
\end{defn}

\begin{lem}
The polynomial $\fc{n}{k}$ is an additive $2$-cocycle.  The space of rational $2$-cocycles is isomorphic to the free $\Q$-module $\Q\{\fc{n}{k} : 0 \le n, k < \infty\}$.
\end{lem}
\begin{proof}
This is a result of Ando, Hopkins, and Strickland~\cite[Proposition A.1]{AHS}.
\end{proof}

\begin{thm}\label{thm:AdditiveStructure}
Write $G_{i, j}$ for the ``gathering operation'' on multi-indices $\lambda$ described by \[G_{i, j}(\lambda) = (\lambda \setminus (\lambda_i, \lambda_j)) \cup (\lambda_i + \lambda_j).\]  Pick a partition $\lambda$ whose entires are all powers of $p$, select an integer $m$, and let $T^m \lambda$ be defined as the set of all $m$-fold gatherings of $\lambda$ (i.e., partitions of the form $G_{i_1, j_1} \cdots G_{i_m, j_m} \lambda$).  Then, if either $\lambda$ is the shortest power-of-$p$ partition of $|\lambda|$ or if $0 \le m < p-1$, the following sum is a cocycle in $\Z/p$: \[ \sum_{\mu \in T^m \lambda} c_\mu \tau(\mu),\] where $c_\mu$ denotes the coefficient of $\tau(\mu)$ in $\fc{|\lambda|}{\ell(\lambda)-m}$.  The cocycles formed in this manner form a basis for the vector space of cocycles in $\Z/p$.
\end{thm}
\begin{thm}\label{thm:GaRepresentingRing}
There is an isomorphism \[C^k(\G_a; \G_m) \cong \spec \Z[c_n \mid n \ge k] \otimes \left(\bigotimes_{\hbox{$p$ prime}} \frac{\Z_{(p)}[b_{p,n,i} \mid 1 \le i < D_{n,k}^p]}{\<pb_{p,n,i}\>} \right),\] where $D_{n, k}^p$ counts the number of power-of-$p$ partitions of $n$ into $k'$ parts, where $k'$ is the smallest possible size equal to or greater than $k$.
\end{thm}
\begin{proof}[Proof]
This is the main result of the previous paper~\cite[Corollary 3.4.10 and Theorem 3.6.2]{HLP}.  However, there is a gap in our proof that we must remark on: the classification there correctly demonstrates this result for $\F_p$-algebras, but does not provide enough to conclude the result for $\Z_{(p)}$-algebras.  To this end, it suffices to show that $\fc{n}{k}$ is the only additive cocycle over $\Z/p^2$ with leading coefficient not divisible by $p$.  The key is that, whereas in characteristic $p$ we have $(a+b)^{p^j} = a^{p^j} + b^{p^j}$, working in $\Z/p^2$ we instead have $(a+b)^{p^j} = a^{p^j} + b^{p^j} + \sum_{i=1}^{p-1} \binom{p^j}{ip^{j-1}} a^{p^{j-1}i} b^{p^{j-1}(p-i)}$, where now $\binom{p^j}{ip^{j-1}}$ is nonzero mod $p^2$.  This has the effect of, in the earlier language, enlarging our annihilator sets dramatically --- namely, any carry minimal partition contains in its annihilator set all other carry-minimal partitions, since we are now able to split and regather power-of-$p$ entries.
\end{proof}

\begin{rem}
We reproduce a diagram from that paper to get a sense of what this scheme looks like.  Leftward arrows denote the gathering operation described above, and rightward arrows denote writing $3^n$ as $3^{n-1} + 3^{n-1} + 3^{n-1}$.
\begin{figure}[htp]
\begin{tikzpicture}[
        normal line/.style={-stealth},
    ]
    \matrix (m) [matrix of math nodes, 
         row sep=2em, column sep=1em,
         text height=1.5ex, 
         text depth=0.25ex]{
                     & \drawbox{\tau(9, 2, 1) - \\
                                \tau(10, 1, 1)} & \drawbox{\tau(9, 1, 1, 1)} &  \\
\drawbox{\tau(9, 3)} &                          & \drawbox{\tau(3, 3, 3, 3)} &                              & \drawbox{\tau(3, 3, 3, 1, 1, 1)} & \cdots \\
                     & \drawbox{\tau(6, 3, 3)}  &                            & \drawbox{\tau(6, 3, 1, 1, 1) \\
                                                                                        \tau(4, 3, 3, 1, 1) \\
                                                                                        \tau(3, 3, 3, 2, 1)} \\
    };
    \path[normal line]
        (m-2-1) edge (m-2-3)
                edge (m-1-3)
        (m-1-3) edge (m-1-2)
                edge (m-2-5)
        (m-2-3) edge (m-3-2)
                edge (m-2-5)
        (m-2-5) edge (m-3-4)
                edge (m-2-6);
\end{tikzpicture}
\caption{The homogeneous degree $12$ part of $C^k(\G_a; \G_m)$ for $2 \le k \le 6$.}
\label{fig:Stratification}
\end{figure}
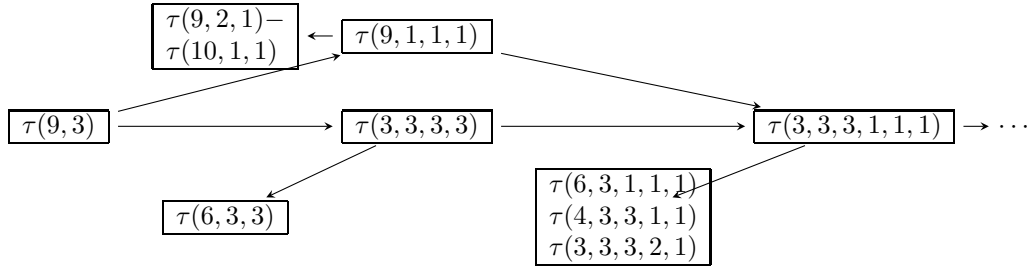
\end{rem}

\section{The tangent spectral sequence}\label{sec:TangentSS}

Now we use the information local to the origin computed in \prettyref{sec:CalcTangent} to produce information about the entire scheme $H^*(\G_a; \G_m)$ through successive approximations.  This procedure assembles into a spectral sequence, and we will specifically investigate a family of nontrivial differentials.

\begin{thm}
There exists a convergent filtration spectral sequence of type \[T_1 H^*(F; G)(R) = H^*(F; \G_a) \Rightarrow H^*(F; G)(R).\]
\end{thm}
\begin{proof}
Let $\sheaf{O}(A^n)$ be $\{\sum_I a_I \mathbf{x}^I \mid \mathbf{x} = (x_1, \ldots, x_n)\}$ the set of $k$-variate power series, which can be identified with the space of scheme-theoretic maps $F^n \to G$ for a pair of $1$-dimensional formal groups $F$ and $G$.  The set $\sheaf{O}(A^n)$ admits a descending filtration by leading degree $d$, denoted \[\sheaf{O}(A^n)_d = \left\{\sum_I a_I \mathbf{x}^I : \hbox{$a_I = 0$ whenever $|I| < d$}\right\}.\]  The inclusions $\sheaf{O}(A^n)_{d+1} \into \sheaf{O}(A^n)_d$ have filtration quotients described by polynomials of homogeneous degree, \[\sheaf{O}(A^n)_d / \sheaf{O}(A^n)_{d+1} \cong \left\{ f \in k[x_1, \ldots, x_n] : f = \sum_{|I| = d} a_I \mathbf{x}^I \right\}.\]  Moreover, this filtration respects the differentials in the complex computing Lubin-Tate cohomology, and so, writing $\sheaf{O}(Z^n)_d$ for the subspace of cocycles of $\sheaf{O}(A^n)_d$, we get a diagram of cochain complexes
\begin{center}
\begin{tikzpicture}[
        normal line/.style={-stealth},
    ]
    \matrix (m) [matrix of math nodes, 
         row sep=2em, column sep=2em,
         text height=1.5ex, 
         text depth=0.25ex]{
\cdots & \sheaf{O}(Z^*)_{d+2} & \sheaf{O}(Z^*)_{d+1} & \sheaf{O}(Z^*)_d & \cdots \\
       & \sheaf{O}(Z^*)_{d+2} / \sheaf{O}(Z^*)_{d+3} & \sheaf{O}(Z^*)_{d+1} / \sheaf{O}(Z^*)_{d+2} & \sheaf{O}(Z^*)_d / \sheaf{O}(Z^*)_{d+1}, \\
    };
    \path[normal line]
        (m-1-1) edge (m-1-2)
        (m-1-2) edge (m-1-3)
                edge (m-2-2)
        (m-1-3) edge (m-1-4)
                edge (m-2-3)
        (m-1-4) edge (m-1-5)
                edge (m-2-4);
\end{tikzpicture}
\end{center}
where each corner is a short exact sequence.  Hence, applying the homology functor produces a convergent spectral sequence of type \[\bigoplus_{d = 0}^\infty H^*\left(\sheaf{O}(Z^*)_d / \sheaf{O}(Z^*)_{d+1}\right) \Rightarrow H^*(F; G).\]  Moreover, because $F(x, y) = x + y \pmod{\< x, y \>^2}$ for all formal group laws $F$, we can identify the cohomologies $H^*(\sheaf{O}(Z^*)_d / \sheaf{O}(Z^*)_{d+1})$ with the subspace of $H^*(F; \G_a)$ consisting of cohomology classes representable as polynomials of homogeneous degree $d$.
\end{proof}

\begin{rem} 
This spectral sequence studies the way to correct a $d$-bud $u$ to a $(d+1)$-bud by adding in a polynomial of homogeneous top degree.  So, the $E_r$-page of the spectral sequence corresponds to classes containing $r$ layers of power series information, which limits to the $E_\infty$-page consisting of the actual power series we want.
\end{rem}

So, the previous calculation $H^*(\G_a; \G_a) = T_1 H^*(\G_a; \G_m)$ serves as input to the tangent spectral sequence.  Because we've seen that $H^*(\G_a; \G_a)(\F_2)$ and $H^*(\G_a; \G_a)(\F_p)$ for $p > 2$ differ, we break into these two cases when studying the spectral sequence.
\begin{thm}
Set $R = \F_2$, $F = \G_a$, and $G = \G_m$, and select $u_+ = c a_i a_j$ for $i \ne j$ and a coefficient $c$.  Then \[d_{2^i+2^j}(c a_i a_j) = c^2(a_i^2 a_{j+1} - a_{i+1} a_j^2).\]
\end{thm}
\begin{proof}
The additive cohomology class $ca_i a_j = c[u]$ can be represented by the polynomial $u = c x^{2^i} y^{2^j}$.  The differentials in the tangent spectral sequence arise from applying the multiplicative coboundary map to bud polynomials, and so we can compute the smallest nonvanishing differential on $c [u]$ by computing \[\frac{1 + cu(x, y)}{1 + cu(w+x, y)} \cdot \frac{1 + cu(w, x+y)}{1 + cu(w, x)},\] provided that the result is not null-cohomologous.  In our case, we have
\begin{align*}
\frac{\left(1 + c x^{2^i} y^{2^j}\right)\left(1 + c w^{2^i}(x + y)^{2^j}\right)}{\left(1 + c (w + x)^{2^i} y^{2^j}\right)\left(1 + c w^{2^i} x^{2^j}\right)} & = \left(1 + k x^{2^i} y^{2^j}\right)\left(1 + k w^{2^i}(x + y)^{2^j}\right) \cdot \\
& \quad\cdot \left(1 - k (w + x)^{2^i} y^{2^j} + k^2 (w + x)^{2^{i+1}} y^{2^{j+1}} +  o(2^{i+1} + 2^{j+1})\right) \cdot \\
& \quad \cdot \left(1 - k w^{2^i} x^{2^j} + k^2 w^{2^{i+1}} x^{2^{j+1}} + o(2^{i+1} + 2^{j+1})\right) \\
& = 1 + k^2 w^{2^i} x^{2^i} y^{2^{j+1}} - k^2 w^{2^{i+1}} x^{2^j} y^{2^j} + o(2^{i+1} + 2^{j+1})).
\end{align*}
Hence, $d_{2^i + 2^j} (c a_i a_j) = c^2(a_i^2 a_{j+1} - a_{i+1} a_j^2)$ as an equation on the $E_{2^i + 2^j}$-page.

To show that this produces a nonvanishing differential, we must show that $a_i^2 a_{j+1} - a_{i+1} a_j^2$ still exists on this page.  The application of $\delta_2$ to $1 + u_+$ has leading additive part of degree divisible by $|u_+|$, hence the first nonvanishing differential with source $a_i^2 a_{j+1} - a_{i+1} a_j^2$ must be on the $E_{2^{i+1}+2^{j+1}}$-page.  To check that it is also not the target of a differential, classes of degree below $2^i + 2^j$ are too far away from $a_i^2 a_{j+1} - a_{i+1} a_j^2$ to hit it with a differential by the $E_{2^i+2^j}$-page.  There is only one class in $E_1^{1, t}$ for $2^i + 2^j \le t < 2^{i+1} + 2^{j+1}$: assuming $i > j$, it is $a_{i+1}$.  We calculate the minimal differential on $a_{i+1}$ similarly as
\begin{align*}
\frac{1+cx^{2^{i+1}}+cy^{2^{i+1}}}{(1+cx^{2^{i+1}})(1+cy^{2^{i+1}})} & = (1+cx^{2^{i+1}} + cy^{2^{i+1}}) \cdot (1 - cx^{2^{i+1}} + c^2 x^{2^{i+2}} - o(3 \cdot 2^{i+1})) \cdot \\
& \hspace{10.2em} \cdot (1 - cy^{2^{i+1}} + c^2y^{2^{i+2}} - o(3 \cdot 2^{i+1})) \\
& = 1 - c^2 x^{2^{i+1}} y^{2^{i+1}} + o(3 \cdot 2^{i+1}).
\end{align*}
Hence $d_{2^{i+1}}(c[a_{i+1}]) = -c^2[a_{i+1}^2]$, and the degree of this class surpasses that of $a_i a_j$.
\end{proof}
\begin{cor}
To extend $c a_i a_j$ to a multiplicative cocycle it is necessary that the coefficient $c$ satisfy $c^2 = 0$.
\end{cor}
\begin{figure}[htp]
\[
\begin{array}{r|cccccc}
8 & a_3 & a_2^2   & a_1^2 a_2        & a_1^4, a_0^2 a_1 a_2 & a_0^2 a_1^3, a_0^4 a_2 & a_0^4 a_1^2 \\
7 &     &         & a_0 a_1 a_2      & a_0^3 a_2, a_0 a_1^2 & a_0^3 a_1^2            & a_0^5 a_1 \\
6 &     & a_1 a_2 & a_0^2 a_2, a_1^3 & a_0^2 a_1^2          & a_0^4 a_1              & a_0^6 \\
5 &     & a_0 a_2 & a_0 a_1^2        & a_0^3 a_1            & a_0^5 \\
4 & a_2 & a_1^2   & a_0^2 a_1        & a_0^4 \\
3 &     & a_0 a_1 & a_0^3 \\
2 & a_1 & a_0^2 \\
1 & a_0 \\
\hline
E_1 & 1 & 2 & 3 & 4 & 5 & 6
\end{array}
\]
\caption{The $E_1$-page of the tangent spectral sequence over $R = \F_2$.}
\end{figure}

\begin{thm}
Set $R = \F_p$, $F = \G_a$, and $G = \G_m$, and select $u_+ = c a_i a_j$ for $i \ne j$ and a coefficient $c$.  Then \[d_{(p-1)(p^i+p^j)}(c a_i a_j) = c^p(a_{i+1} b_{j+1} - a_{j+1} b_{i+1}).\]
\end{thm}
\begin{proof}[Proof] 
This is not relevant to our main result, so we will only give a sketch.  Use the truncated exponential function to build a representing polynomial \[\operatorname{texp}(c x^{p^i} y^{p^j}) = \sum_{n=0}^{p-1} \frac{(c x^{p^i} y^{p^j})^n}{n!}\] for the cohomology class $c a_i a_j$.  Then, there is an equality $\operatorname{texp}(z)^{-1} \equiv \operatorname{texp}(-z) \mod{z^{p+1}}$, and so to calculate the $c^p$ coefficient of $\delta_2 \operatorname{texp}(c x^{p^i} y^{p^j})$ we may calculate the $c^p$ coefficient of the product \[\operatorname{texp}(c x^{p^i} y^{p^j}) \cdot \operatorname{texp}(-c (w+x)^{p^i} y^{p^j}) \cdot \operatorname{texp}(c w^{p^i} (x+y)^{p^j}) \cdot \operatorname{texp}(-c w^{p^i} x^{p^j}).\]  Brutal expansion shows that this is $w^{p^{i+1}} \fc{p^{j+1}}{2} - \fc{p^{i+1}}{2} y^{p^{j+1}}$.
\end{proof}
\begin{figure}[htc]
\[
\begin{array}{r|cccccc}
9 & a_2 & b_2     &                  &             & a_1 b_1^2                & b_1^3 \\
8 &    &          &                  &             &                          & a_0 a_1 b_0 b_1 \\
7 &     &         &                  & a_0 a_1 b_1 & a_1 b_0 b_1, a_0 b_1^2   & b_0 b_1^2 \\
6 &     &         & a_1 b_1          & b_1^2       &                          & a_0 a_1 b_0^2 \\
5 &     &         &                  & a_0 a_1 b_0 & a_0 b_0 b_1              & b_0^2 b_1 \\
4 &     & a_0 a_1 & a_1 b_0, a_0 b_1 & b_0 b_1     \\
3 & a_1 & b_1     &                  &             & a_0 b_0^2                & b_0^3 \\
2 &     &         & a_0 b_0          & b_0^2       \\
1 & a_0 & b_0     \\
\hline
E_1 & 1 & 2 & 3 & 4 & 5 & 6
\end{array}
\]
\caption{The $E_1$-page of the tangent spectral sequence over $R = \F_3$.}
\end{figure}

\begin{cor}\label{cor:TangentSSUpshot}
A modular additive cocycle in characteristic $p$ represented by the cohomology class $c[a_i a_j]$ will not occur as a leading summand of a multiplicative cocycle unless $c^p = 0$.
\end{cor}

\begin{rem}
Note that these results obstruct particular asymmetric cohomology classes, and that the symmetric cocycles $\fc{n}{2}$ do not support these differentials.  For example, the cohomology class $[\fc{p^i+p^j}{2}]$ is represented as $a_ia_j + a_ja_i = a_ia_j - a_ia_j = 0$, which has no obstruction.
\end{rem}

\section{Half-Weil forms}

\begin{defn}
Fix a group scheme $G$ and an integer $n$.  For any $1 \le i \le n$, define $p_i$ to be the map $G^n \to G^n$ defined by $1 \times \cdots \times 1 \times p \times 1 \times \cdots \times 1$, with the $p$ occuring in the $i$th position.
\end{defn}

\begin{thm}
Select a trivialized multiextension $\torsor L$ with controlling cocycles $u_1, \ldots, u_n$.  There is a diagram
\begin{center}
\begin{tikzpicture}[
        dash line/.style={densely dotted},
        normal line/.style={-stealth},
    ]
    \matrix (m) [matrix of math nodes, 
         row sep=4em, column sep=4em,
         text height=1.5ex, 
         text depth=0.25ex]{
p_* \torsor{L} & \torsor{L}^{\otimes p} & p_i^* \torsor{L} \\
\torsor{L} & & \torsor{L}, \\
    };
    \path[normal line]
        (m-2-1) edge node[left]{$p_*$} (m-1-1)
                edge node[above]{$\otimes p$} (m-1-2)
                edge node[above]{$p \cdot -$} (m-2-3)
        (m-1-3) edge node[right]{$p_i^*$} (m-2-3)
        (m-1-2) edge [dash line] node[above]{$\beta_i$} (m-1-3)
                edge node[above]{\hspace{2em} $\mu_i^{\circ(p-1)}$} (m-2-3)
        (m-1-1) edge [dash line] node[above]{$\alpha$} (m-1-2);
\end{tikzpicture}
\end{center}
factoring the multiplication-by-$p$ map on $\torsor L$ so that the composition of the top row is an isomorphism of torsors.
\end{thm}
\begin{proof}
Given the trivialization of $\torsor L$, this proof is completely computational.  The map $\otimes p$ is described by the formula \[\otimes p: \left( g \times (x_1, \ldots, x_n) \right) \mapsto \left(g \times (x_1, \ldots, x_n)\right)^{\otimes p} = \left( g^p \times (x_1, \ldots, x_n) \right).\]  Then, the map $p_*$ acts as \[p_*: (g \times (x_1, \ldots, x_n)) \mapsto (g^p \times (x_1, \ldots, x_n)),\] which determines the map $\alpha$ to be \[\alpha: (g \times (x_1, \ldots, x_n)) \mapsto (g \times (x_1, \ldots, x_n)).\]  Hence, $\alpha$ is an isomorphism.

We perform the same analysis on $\beta_i$.  The map $p_i^*$ acts by \[p_i^*: (g \times (x_1, \ldots, x_n)) \mapsto g \times (x_1, \ldots, x_{i-1}, px_i, x_{i+1}, \ldots, x_n).\]  Then, the map $\mu_i^{\circ (p-1)}$ acts by iterated multiextension addition in the $i$th factor, which gives the formula \[\mu_i^{\circ (p-1)}: (g \times (x_1, \ldots, x_n)) \mapsto \left( g \prod_{i=1}^{p-1} u_i(x_1, \ldots, \widehat{x_i}, \ldots, x_n)(x_i, ix_i), (x_1, \ldots, px_i, \ldots, x_n)\right).\]  Hence, the map $\beta_i$ is determined to be \[\beta_i: (g \times (x_1, \ldots, x_n)) \mapsto \left( g \prod_{i=1}^{p-1} u_i(x_1, \ldots, \widehat{x_i}, \ldots, x_n)(x_i, ix_i), (x_1, \ldots, x_n) \right).\]  Because the twist in the $G$-factor is invertible, as $G$ is a group, $\beta_i$ is also an isomorphism.
\end{proof}

\begin{defn}\label{def:HalfWeilPairing}
Given a $k$-variate multiplicative $2$-cocycle $u \in C^k(\G_a; \G_m)$, we define the associated ``half-Weil form'' \[e = \prod_{i=1}^{p-1} u(i x_1, x_1, x_2, \ldots, x_{k-1}).\]
\end{defn}

\begin{thm}
The half-Weil form $e$ associated to such a $u \in C^k(\G_a; \G_m)(R)$ for $R$ an $\F_p$-algebra is a (not necessarily symmetric) $(k-1)$-variate multiplicative $2$-cocycle satisfying \[\delta_1 e = u^p.\]
\end{thm}
\begin{proof}
As $p = 0$ in $R$, the rigidified higher cubical structure $\torsor B$ associated to $u$ has trivial pullback $p_i^* \torsor B$, since the cocycles associated to $p_i^* \torsor B$ are of the two forms \[
u_j(x_1, \ldots, x_{i-1}, 0, x_{i+1}, \ldots, \widehat{x_j}, \ldots, x_n)(x_j, x_j') = 1,\] \[u_i(x_1, \ldots, \widehat{x_i}, \ldots, x_n)(0, 0) = 1.\]  The isomorphism $p_i^* \torsor B \cong p_* \torsor B$ can be reinterpreted as a trivialization of $p_* \torsor B / p_i^* \torsor B$, i.e., a $1$-cocycle whose image under $\delta_1$ is the $2$-cocycle associated to $p_* \torsor B / p_i^* \torsor B$ --- but, since the $2$-cocycle associated to $p_i^* \torsor B$ is $1$, we're really trivializing $p_* \torsor B$, which has associated $2$-cocycle $\left(u(x_1, \ldots, x_n)\right)^p$.

We can produce an explicit formula for this $1$-cocycle by chasing points around the diagram
\begin{center}
\begin{tikzpicture}[
        dash line/.style={densely dotted},
        normal line/.style={-stealth},
    ]
    \matrix (m) [matrix of math nodes, 
         row sep=2em, column sep=2em,
         text height=1.5ex, 
         text depth=0.25ex]{
p_* \torsor B \otimes p_* \torsor B & p_* \torsor B \\
p_i^* \torsor B \otimes p_i^* \torsor B & p_i^* \torsor B. \\
    };
    \path[normal line]
        (m-1-1) edge node[above]{$\mu_j$} (m-1-2)
                edge node[left]{$\beta\alpha \otimes \beta\alpha$} (m-2-1)
        (m-1-2) edge node[right]{$\beta\alpha$} (m-2-2)
        (m-2-1) edge node[above]{$\mu_j$} (m-2-2);
\end{tikzpicture}
\end{center}
We make the following two computations, writing $\mathbf{x} = (x_1, \ldots, x_j, \ldots, x_n)$ and $\mathbf{x}' = (x_1, \ldots, x_j', \ldots, x_n)$:
\begin{align*}
& \beta\alpha \circ \mu_j (g \times \mathbf{x}) \otimes (g' \times \mathbf{x}') = \\
& = \beta\alpha (g g' u(x_1, \ldots, \widehat{x_j}, \ldots, x_n)(x_j, x_j')^p) \times (x_1, \ldots, x_j + x_j', \ldots, x_n) \\
& = (g g' u(x_1, \ldots, \widehat{x_j}, \ldots, x_n)(x_j, x_j')^p e(x_1, \ldots, x_j + x_j', \ldots, x_n)) \times (x_1, \ldots, x_j + x_j', \ldots, x_n), \\ \\
& \mu_j (\beta\alpha \otimes \beta\alpha) (g \times \mathbf{x}) \otimes (g' \times \mathbf{x}') = \\
& = \mu_j (g e(\mathbf{x}) \times \mathbf{x}) \otimes (g' e(\mathbf{x}') \times \mathbf{x}') \\
& = (g g' u(x_1, \ldots, px_i, \ldots, \widehat{x_j}, \ldots, x_n)(x_j, x_j') e(\mathbf{x}) e(\mathbf{x}')) \times (x_1, \ldots, x_j + x_j', \ldots, x_n).
\end{align*}
Because $\beta\alpha$ is a map of multiextensions, the $G$-coordinates of these expressions must be equal, and hence \[u(x_1, \ldots, \widehat{x_j}, \ldots, x_n)(x_j, x_j')^p = \frac{e(x_1, \ldots, x_j, \ldots, x_n) e(x_1, \ldots, x_j', \ldots, x_n)}{e(x_1, \ldots, x_j + x_j', \ldots, x_n)} = \delta_1 e(x_1, \ldots, x_n). \qedhere\]
\end{proof}

\begin{rem}
The classical Weil pairing associated to a cubical structure arises as the composite isomorphism \[(p \times 1)^* \torsor L \xrightarrow{\alpha^{-1} \beta_1^{-1}} p_* \torsor L \xrightarrow{\beta_2 \alpha} (1 \times p)^* \torsor L,\] which in the fiber over a point $(x_1, x_2) \in \G_a^2$ acts by multiplication by \[\prod_{i=1}^{p-1} \frac{u(x_1, ix_1, x_2)}{u(x_1, ix_2, x_2)},\] where $u$ is the $2$-cocycle associated to $\torsor L$.  This was the object used by Mumford~\cite{Mumford} in his comparison of Weil pairings and cubical structures.
\end{rem}

\begin{lem}\label{lem:AdditiveWeilForm}
There is an additive version of the half-Weil form.  To a multiplicative $2$-cocycle $u$ with additive part $u_+$, we associate an additive half-Weil form $e_+$, which is determined by \[e_+ = \sum_{i=1}^{p-1} u_+(ix_1, x_1, x_2, \ldots, x_{k-1})\] when this sum is nonzero.  Again, $\delta_1 e_+ = p u_+ \equiv 0$.
\end{lem}
\begin{proof}
Reuse the above argument for multiextensions with structure group $\G_a$ rather than $\G_m$ to produce an additive notion of Weil pairing.  The statement about the interaction of $u_+$ and $e_+$ and $u$ and $e$ stems from studying the filtration on the tangent space used in \prettyref{sec:TangentSS}.
\end{proof}

\begin{lem}\label{lem:WeilCalculation}
The sum given above determining the additive half-Weil form associated to $\fc{n}{2}$ is $-\fc{n}{1}$ when $n = p^i$ and $0$ otherwise.
\end{lem}
\begin{proof}
Setting $u_+ = (x+y)^n - x^n - y^n$ as an additive $2$-cocycle over $\Z$, we telescope and calculate \[e_+ = \sum_{j=1}^{p-1} x^n \left( (j+1)^n - j^n - 1 \right) = p x^n \left( p^{n-1} - 1 \right).\]  Then, $\fc{p^i}{2} = p^{-1} ((x+y)^{p^i} - x^{p^i} - y^{p^i})$, and hence the associated half-Weil form over $\F_p$ is $-\fc{p^i}{1}$.  But, when $n$ is not of the form $p^i$, $\fc{n}{2} = (x+y)^n - x^n - y^n$ exactly, and hence reducing modulo $p$ gives $0$.
\end{proof}

\section{Obstructions and the calculation of $C^k(\G_a; \G_m) \times \spec \Z_{(2)}$}

\begin{thm}\label{thm:ArtinHasseExtension}
Write $\nu_p(n)$ for the order of $p$-divisibility of the integer $n$, and recall the function $\phi(n, k)$ from \prettyref{def:zeta}.  Let $E_p(t)$ be the Artin-Hasse exponential, a $p$-integral series defined by \[E_p(t) = \exp \left( \sum_{k=0}^\infty \frac{t^{p^k}}{p^k} \right).\]  Then, when $\nu_p \phi(n, k) < \nu_p(n)$, the power series $\tilde{\fc{n}{k}} = (\delta^1)^{\circ(k-1)} E_p(c x^n)^{p^{-\nu_p \phi(n, k)}}$ is a multiplicative extension of $c \fc{n}{k}$ over an $\F_p$-algebra.
\end{thm}
\begin{proof}
See Ando, Hopkins, and Strickland~\cite[Corollary 3.22]{AHS}.
\end{proof}

\begin{rem}
Ando, Hopkins, and Strickland also provide the equation \[\nu_p \phi(n, k) = \max \left\{ 0, \left\lceil \frac{k-\sigma_p(n)}{p-1} \right\rceil\right\}\] to aid in computing facts about this power series, where $\sigma_p(n)$ is the $\N$-valued digital sum of $n$ in radix $p$.  This is an immediate consequence of work of K\"ummer~\cite{Kummer}.
\end{rem}

\begin{thm}\label{thm:WeilObstructions}
Every additive cocycle $u_+$ over a ring $S$ of characteristic $p$ can be written in the form
\[
u_+ = \sum_{\substack{n, m \\ \ell(I) = k-3}} r_{n, m, I} \fc{n}{2} x_3^m (x_4, \ldots, x_k)^I
,\]
where $r_{n, m, I}$ is an element in $S$.  If $r_{p^n, p^m, I} \ne r_{p^m, p^n, I}$ for any choice of    $m$, $n$, and $I$, then any multiplicative $2$-cocycle $1 + b u_+ + o(|u_+|)$ must satisfy $b^p = 0$.
\end{thm}
\begin{proof}
Select such a cocycle $u$, along with indices $n$, $m$, and $I$ so that $r_{p^n, p^m, I} \ne r_{p^m, p^n, I}$ and assume $r_{p^n, p^m, I} \ne 0$.  Construct the associated half-Weil pairing $e$ as in \prettyref{def:HalfWeilPairing}; by assumption and \prettyref{lem:WeilCalculation} $e_+$ is nonzero and the projection of the cohomology class $[e_+] \in H^2(\G_a; \G_a)(S[x_3, \ldots, x_{k-1}])$ onto the module factor $S \{ a_n a_m \}$ is nonzero with coefficient $r_{p^n, p^m, I} - r_{p^m, p^n, I} \ne 0$.  Hence, the tangent spectral sequence and \prettyref{cor:TangentSSUpshot} dictate that the leading coefficient of $e$ must have zero $p$th power, as $e$ is a multiplicative extension of $e_+$.  By distributivity and the characterization of $e_+$ in \prettyref{lem:AdditiveWeilForm}, this coefficient is an integer multiple of $b$.
\end{proof}

We need a small lemma to interrelate these obstructions before we can perform the promised $2$-primary calculation:
\begin{lem}
Fix an integer $n$ and construct a graph whose nodes are labeled by unordered tuples of powers of $2$ whose sum is $n$, and insert an edge from a node of tuple length $\ell$ to a node of tuple length $\ell-1$ if exactly two entries of $\ell$ can be summed together to produce the second tuple.  Every subgraph consisting of all nodes of lengths $\ell$ and $\ell-1$ is connected, i.e., for any two tuples of length $\ell$, we can find a path in the graph connecting them.
\end{lem}
\begin{proof}
Associate to such a tuple $\lambda$ the finite sequence of naturals $c^\lambda_n$ so that $c^\lambda_n$ counts the number of times $2^n$ appears in $\lambda$, and order the set of tuples $\lambda$ by the dictionary order on the associated sequences $c^\lambda_n$.  Given any two nonequal tuples $\lambda$ and $\lambda''$ of length $\ell$, we can assume $\lambda > \lambda''$; we want to construct a tuple $\lambda'$ with $\lambda > \lambda' > \lambda''$ by following edges in the graph.  An edge from $\lambda$ of length $\ell$ to $\mu$ of length $\ell - 1$ corresponds in sequences to \[c^\mu_n = \begin{cases} c^\lambda_n & n \ne i, \\ c^\lambda_i - 2 & n = i, \\ c^\lambda_{i+1} + 1 & n = i+1\end{cases}\] for some selected index $i$.  So, as $\lambda < \lambda''$, we select the first differing index $i$ and remove $2$ from $c^\lambda_i$, add $1$ to $c^\lambda_{i+1}$, then select any index $j > i$ and remove $1$ from $c^\lambda_j$, adding $2$ to $c^\lambda_{j-1}$.  The resulting sequence describes a new tuple $\lambda'$ satisfying $\lambda < \lambda' < \lambda''$.  Induction on the imposed ordering gives the lemma.\footnote{This is a ghost of the argument used in \prettyref{thm:GaRepresentingRing}.}
\end{proof}
\begin{cor}
If $u_+$ is an additive $2$-cocycle over $\F_2$, then if $u_+ \ne \fc{n}{2}$ for some $n$, $u_+$ is obstructed by \prettyref{thm:WeilObstructions}.
\end{cor}
\begin{proof}
Every obstruction $a_ia_j$ stemming from an application of \prettyref{thm:WeilObstructions} to a term of the form $\fc{2^i}{2} x_3^{2^j}$ can only be canceled by the appearance of a term of the form $\fc{2^j}{2} x_3^{2^i}$, and hence the entire connected component of the graph in the above lemma containing any of the tuples appearing in $u_+$ must appear, lest we produce a nontrivial obstruction.  The lemma says the graph itself is connected, hence $u_+$ must be a scalar multiple of $\fc{|u_+|}{2}$.
\end{proof}

\begin{lem}\label{lem:fcObstructions}
For $\nu_2 \phi(n, k) > \nu_2 n$, $u_+ = \fc{n}{2}$ is obstructed from extending to a multiplicative $2$-cocycle.
\end{lem}
\begin{proof}
The case $\nu_2 \phi(n, k) > \nu_2 n$ corresponds exactly to the appearance of summands of the form $\tau(2^{i_1}, \ldots, 1)$ in $\fc{n}{k}$.  Applying \prettyref{thm:WeilObstructions}, we produce an obstruction of the form $a_{i_1} a_0$ with no mirror $a_0 a_{i_1}$, since $[a_0]$ is not in the image of the additive Weil pairing $e_+$.
\end{proof}

\begin{thm}\label{thm:BigF2Thm}
We compute
\begin{align*}
\sheaf{O}(\spec \F_2 \times C^k(\G_a; \G_m)) & = \F_2[z_n \mid \nu_2 \phi(n, k) \le \nu_2 n] \otimes \\
& \otimes \Gamma[b_{n, \gamma_2(n, k)} \mid \nu_2 \phi(n, k) > \nu_2 n] \otimes \\
& \otimes \F_2[b_{n, i} \mid \gamma_2(n, k) < i < D_{n, k}] / \<b_{n,i}^2\>,
\end{align*}
where $n \ge k$ ranges over integers, $D_{n, k}$ is the coefficient of the generating function \[\prod_{i=0}^\infty \frac{1}{1 - tx^{2^i}} = \sum_{n, k} D_{n, k} x^n t^k,\] and $\gamma_p(n, k) = \max\{0,\min\{k-\sigma_p(n),\nu_p(n)\}\}$ counters the number of divided power classes introduced already.
\end{thm}
\begin{proof}
These tensor factors correspond, in order, to the additive cocycles $\fc{n}{k}$ which extend freely, to the additive cocycles $\fc{n}{k}$ which are obstructed by \prettyref{lem:fcObstructions}, and to the remaining modular additive cocycles $\tau(\lambda)$ not already belonging to a divided power structure, as these are also obstructed by \prettyref{thm:WeilObstructions}.
\end{proof}

\begin{cor}\label{cor:BigZ2Thm}
We compute the $2$-primary component to be
\begin{align*}
\sheaf{O}(\spec \Z_{(2)} \times C^k(\G_a; \G_m)) & = \Z_{(2)}[z_n \mid \nu_2 \phi(n, k) \le \nu_2 n] \otimes \\
& \otimes \Gamma[b_{n, \gamma_2(n, k)} \mid \nu_2 \phi(n, k) > \nu_2 n] \otimes \\
& \otimes \Z_{(2)}[b_{n, i} \mid \gamma_2(n, k) < i < D_{n, k}] / \<2b_{n,i}, b_{n,i}^2\>.
\end{align*}
\end{cor}
\begin{proof}
This follows immediately from \prettyref{thm:BigF2Thm}, which gives the general structure of the answer, and \prettyref{thm:GaRepresentingRing}, which shows that the $j$ in the $2^j b_{n, i}$ in the quotient must be a $1$.
\end{proof}

\begin{rem}
Using the structure of \prettyref{thm:AdditiveStructure} in the low dimensional case of $k = 3$, we recover for all primes $p$ the original computation of $C^3(\G_a; \G_m)$ of Ando, Hopkins, and Strickland~\cite[Section 3]{AHS}.
\end{rem}

\section{Outro}

The study of this scheme was motivated by a paper of Ando, Hopkins, and Strickland~\cite{AHS}, and so now we reflect on its relation to their work, its relation to topology, and what's left to pin down.

\begin{rem}
The core of their work is to compare the functors $\spec E_0 BU\<2k\>$ and $C^k(\spf E^0 \CP^\infty; \G_m)$ for certain (co)homology theories $E$, where $BU\<2k\>$ denotes the $(2k-1)$-connected cover of $BU \times \Z$, the representing space for complex $K$-theory.  The multiplicative structure on $E_0 BU\<2k\>$ arises from the destabilization of the Whitney sum of stable virtual bundles, and in the case that $E_0 BU\<2k\>$ is even-concentrated, $\spec E_0 BU\<2k\>$ makes sense.  They reduce to the cases $E = H\Q$ and $E = H\F_p$, where $\spf E^0 \CP^\infty \cong \G_a$, and they complete the proof by explicitly calculating $\sheaf{O}(C^k(\G_a; \G_m))$ for $k \le 3$ and recalling the previously known calculation of $H^*(BU\<2k\>; \F_p)$ due to Singer~\cite{Singer}.

The computation in this paper is an attempt to compare $\sheaf{O}(C^k(\G_a; \G_m))$ and $H_* BU\<2k\>$ for $k > 3$, where we immediately run into trouble.  Singer's calculation describes $H^* BU\<2k\>$ as a quotient of $H^* BU$ tensored together with a certain subalgebra of $H^* K(\Z, 2k-3)$ which contains the class $\operatorname{Sq}^7 \operatorname{Sq}^3 \iota_{2k-3}$ for $k > 3$, which is of odd cohomological degree.  The usual supercommutativity present in algebraic topology presents an obstacle to the immersion of the ring $H_* BU\<2k\>$ into algebraic geometry, which traditionally takes as input only commutative rings, and so we must modify what ring we expect to compare to $\sheaf{O}(C^k(\G_a; \G_m))$.

Calculational experiments with Mathematica have shown that the graded ranks of indecomposables in $H^*(BU\<2k\>; \F_2)$ match those of the indecomposables in $\sheaf{O}(\spec \F_2 \times C^k(\G_a; \G_m))$ through some $240$ bidegrees after we delete the closure of the odd dimensional classes under the action of the Steenrod algebra.
\end{rem}

\begin{rem}
The construction of the map $\spec H_*(BU\<2k\>; \F_2) \to C^k(\G_a; \G_m) \times \spec \F_2$ described by Ando, Hopkins, and Strickland admits a certain compatibility with the Steenrod algebra suggested to be present by the above brute-force computation.  The module $H_*(BU\<2k\>; \F_2)$ is a coalgebra over the dual Steenrod algebra almost by definition, and the scheme $C^k(\G_a; \G_m)$ carries an action of the scheme $\underline{\operatorname{Aut}}(\G_a)$ of group automorphisms of the additive formal group.  Long-standing work identifying the role of cohomology operations / homology cooperations in the context of chromatic homotopy theory has shown that the dual Steenrod algebra occurs as the ring of functions on $\underline{\operatorname{Aut}}(\G_a)$, and hence the coaction of the dual Steenrod algebra on $H_*(BU\<2k\>; \F_2)$ can be seen as an action of $\underline{\operatorname{Aut}}(\G_a)$ on $\spec H_*(BU\<2k\>; \F_2)$.  Moreover, the map $\spec H_*(BU\<2k\>; \F_2) \to C^k(\G_a; \G_m) \times \spec \F_2$ is seen to be $\underline{\operatorname{\Aut}}(\G_a)$-equivariant.

To use their map to form the comparison of $H_*(BU\<2k\>; \F_2)$ and $C^k(\G_a; \G_m) \times \spec \F_2$ for $k > 3$, we need to be able to describe it in some detail, and its equivariance with respect to this action greatly rigidifies it, provided we can calculate the $\underline{\operatorname{Aut}}(\G_a)$-action on both of these objects.  The description of the action on the nilpotent part of $C^k(\G_a; \G_m) \times \spec \F_2$ is quite easy to calculate, but the action on the free part is not known at this time.
\end{rem}

\begin{rem}
The Adams splitting of the connective $K$-theory spectrum $ku$ is an important structural fact in stable homotopy theory.  There is a spectrum $BP$ occuring as the minimal summand in the $p$-localization of the complex bordism spectrum $MU$, and Wilson~\cite{Wilson} describes a sequence of approximating spectra \[BP\<\infty\> \to \cdots \to BP\<n\> \to \cdots \to BP\<1\> \to BP\<0\>,\] with $BP\<\infty\> \simeq BP$, $BP\<0\> \simeq H\Z_{(p)}$, and $\pi_* BP\<n\> \cong \Z_{(p)}[v_1, \ldots, v_n]$ with $|v_n| = 2(p^n-1)$.  The folk theorem states that as ring spectra we have a splitting \[L_p ku \simeq \bigvee_{i=0}^{p-2} \Sigma^{2i} BP\<1\>.\]

That is, the data in connective $K$-theory falls neatly into bands described by these truncated Brown-Peterson summands.  In the previous paper~\cite{HLP}, as seen in part in \prettyref{fig:Stratification}, we witnessed a similar banding in the data, described in the $0$th stratum by power-of-$p$ multi-indices and in the $n$th stratum by distance leftward (i.e., in decreasing dimension) from the power-of-$p$ band.  Something similar happens in our \prettyref{thm:WeilObstructions}; it's a necessary hypothesis that we be working in the band one step leftward of the power-of-$p$ band, otherwise the obstruction produced by the half-Weil pairing is always $0$.

It would be interesting (and likely important) to understand what subfunctor $\spec H_* \Omega^{\infty-k} BP\<1\>$ represents and what of our methods are more appropriately cast in that language.  In fact, it is an interesting question what $\spec H_* \Omega^{\infty-k} BP\<k'\>$ represents in general, and how these are assembled from the even further split objects $\spec H_* Y_k$.
\end{rem}

\begin{rem}
One idea unexploited in this paper is Cartier duality.  For an even-concentrated $H$-space $X$ and even periodic ring spectrum $E$, both the homology $E_0 X$ and cohomology $E^0 X$ are Hopf algebras, and the duality between their multiplications and diagonals is encoded in the algebro-geometric formula $\underline{\Hom}(\spf E^0 X, \G_m) \cong \spec E_0 X$.  In general, the object $\underline{\Hom}(\spf E^0 X, \G_m)$ is called the Cartier dual of the group scheme $\spf E^0 X$.  Our calculation in \prettyref{thm:BigF2Thm} demonstrates that $\spec \Z_{(2)} \times C^k(\G_a; \G_m)$ has a well-behaved Cartier dual $C_{k,(2)}(\G_a)$ satisfying $\underline{\Hom}(C_{k,(2)}(\G_a), \G_m) \cong C^k(\G_a; \G_m) \times \spec \Z_{(2)}$, and we expect these congruences to match up in the sense that the following diagram should commute:
\begin{center}
\begin{tikzpicture}[
        normal line/.style={-stealth},
    ]
    \matrix (m) [matrix of math nodes, 
         row sep=2em, column sep=3em,
         text height=1.5ex, 
         text depth=0.25ex]{
\spec H_*(BU\<2k\>; \F_2) & \spec \F_2 \times C^k(\G_a; \G_m) \\
\underline{\Hom}(\spf H^*(BU\<2k\>; \F_2), \G_m) & \underline{\Hom}(C_{k,(2)}(\G_a), \G_m). \\
    };
    \path[normal line]
        (m-1-1) edge (m-1-2)
                edge[style={double,-}] (m-2-1)
        (m-1-2) edge[style={double,-}] (m-2-2)
        (m-2-1) edge (m-2-2);
\end{tikzpicture}
\end{center}

Then, because $ku$ is a ring spectrum and $H\F_2$ has K\"unneth isomorphisms, we should expect that $H_*(BU\<2k\>; \F_2)$ assemble into a Hopf ring as $k$ varies.  The induced structure on formal schemes is harder to understand; that $\spf$ and $\spec$ are arrow-reversing indicates that $\spf H_*(BU\<2*\>; \F_2)$ should assemble into a ``coring scheme,'' a somewhat unfamiliar object.  However, the dual schemes $\spf H^*(BU\<2*\>; \F_2)$ assemble into a graded ring scheme, and using Cartier duality, understanding these objects should in turn give descriptions of the original homological objects of interest.  This program is outlined in part by Ando, Hopkins, and Strickland~\cite[Remark 2.32]{AHS} in their original paper.
\end{rem}

\bibliography{mult-cocycles}{}
\bibliographystyle{plain}

\end{document}